\documentclass[12pt,twoside]{amsart}
\usepackage{amsmath}
\usepackage{amsthm}
\usepackage{amsfonts}
\usepackage{amssymb}
\usepackage{latexsym}
\usepackage{mathrsfs}
\usepackage{amsmath}
\usepackage{amsthm}
\usepackage{amsfonts}
\usepackage{amssymb}
\usepackage{latexsym}
\usepackage{geometry}
\usepackage{dsfont}
\usepackage[dvips]{graphicx}
\usepackage{color}
\usepackage[all]{xy}

\date{}
\allowdisplaybreaks[4] \footskip=15pt
\renewcommand{\uppercasenonmath}[1]{}

\usepackage{graphicx,amssymb}
\usepackage[all]{xy}
\usepackage{amsmath}

\allowdisplaybreaks
\usepackage{amsthm}
\usepackage{color}

\theoremstyle{plain}
\newtheorem{theorem}{Theorem}[section]
\newtheorem{proposition}[theorem]{Proposition}
\newtheorem{lemma}[theorem]{Lemma}
\newtheorem{corollary}[theorem]{Corollary}

\newtheorem*{open question}{Open Question}
\newtheorem{definition}[theorem]{Definition}

\theoremstyle{definition}

\theoremstyle{remark}
\newtheorem{remark}[theorem]{Remark}
\newcommand{\ra}{\rightarrow}

\newcommand{\Tor}{\mbox{\rm Tor}}

\newcommand{\Q}{\mathcal{Q}}

\newcommand{\prodi}{\prod_{i\in I}}
\newcommand{\prodij}{\prod_{i\in I,j\in J}}

\newcommand{\wMLR}{\ensuremath{S\mbox{-}\{R\}\mbox{-}\mathrm{ML}}}
\newcommand{\wMLF}{\ensuremath{S\mbox{-}\mathcal{F}\mbox{-}\mathrm{ML}}}
\newcommand{\wMLP}{\ensuremath{S\mbox{-}\mathcal{P}\mbox{-}\mathrm{ML}}}
\newcommand{\wMLQ}{\ensuremath{S\mbox{-}\mathcal{Q}\mbox{-}\mathrm{ML}}}

\newcommand{\MLPsubQ}{\ensuremath{S\mbox{-}Psub\mathcal{Q}\mbox{-}\mathrm{ML}}}
\newcommand{\MLlimQ}{\ensuremath{S\mbox{-}\mathcal{\lim\limits_{\longrightarrow}Q}\mbox{-}\mathrm{ML}}}
\newcommand{\MLProdQ}{\ensuremath{S\mbox{-}Prod\mathcal{Q}\mbox{-}\mathrm{ML}}}
\newcommand{\MLcQ}{\ensuremath{S\mbox{-}\overline{\mathcal{Q}}\mbox{-}\mathrm{ML}}}

\def\ra{\rightarrow}

\def\Tor{{\rm Tor}}

\def\Ker{{\rm Ker}}

\def\Im{{\rm Im}}
\def\Coker{{\rm Coker}}

\begin{document}
\begin{center}
{\large  \bf New Characterizations of $S$-coherent rings}

\vspace{0.5cm}   Wei Qi$^{a}$,\ Xiaolei Zhang$^{b}$,\ Wei Zhao$^{c}$
%\bigskip

{\footnotesize a.\ School of Mathematical Sciences, Sichuan Normal University, Chengdu 610068, China\\
b.\ Department of Basic Courses, Chengdu Aeronautic Polytechnic, Chengdu 610100, China\\
c.\ School of Mathematics, ABa Teachers University, Wenchuan 623002, China\\

E-mail: zxlrghj@163.com\\}
\end{center}
%\begin{figure}[b]
%\rule[-2.5truemm]{5cm}{0.1truemm}\\[2mm]
%{\small }
%\end{figure}

%\begin{figure}[b]
%\rule[-2.5truemm]{5cm}{0.1truemm}\\[2mm]
%{\small }
%\end{figure}
\bigskip
\centerline { \bf  Abstract}
\bigskip
\leftskip10truemm \rightskip10truemm \noindent

In this paper, we introduce and study the class $S$-$\mathcal{F}$-ML of $S$-Mittag-Leffler modules with respect to all flat modules. We show that a ring $R$ is $S$-coherent if and only if $S$-$\mathcal{F}$-ML is closed under submodules. As an application, we obtain the $S$-version of Chase Theorem: a ring $R$ is $S$-coherent if and only if any direct product of $R$ is $S$-flat if and only if any direct product of flat $R$-modules is $S$-flat. Consequently, we provide an answer to the open question proposed by D. Bennis and M. El Hajoui \cite{bh18}.
\vbox to 0.3cm{}

\leftskip0truemm \rightskip0truemm
\bigskip
%\section { \bf Introduction    }
%\bigskip

\section{Introduction}
Throughout this paper, $R$ is a commutative ring with identity and all modules are unitary. $S$ will always denote a multiplicative closed set of $R$. In the past few years, $S$-versions of some classical notions have been studied by many authors. In 2002, D. D. Anderson and T. Dumitrescu \cite{ad02} introduced $S$-finite modules and $S$-Noetherian rings and extended the classical Cohen's Theorem and Hilbert basis Theorem to some $S$-versions. In 2014, H. Kim, M. O. Kim and J. W. Lim \cite{kkl14} introduced $S$-strong Mori domains and proved that if $S$ is an anti-archimedean subset of a domain $D$, then $D$ is an $S$-strong Mori domain if and only if the polynomial ring $D[X]$ is an $S$-strong Mori domain, if and only if the $t$-Nagata ring $D[X]N_v$ is an $S$-strong Mori domain, if and only if $D[X]N_v$ is an $S$-Noetherian domain.
In 2015, J. W. Lim \cite{l15} studied Nagata ring of $S$-Noetherian domains and locally $S$-Noetherian domains and proved that  if $S$ is an anti-archimedean subset of a domain $D$, then $D$ is an $S$-Noetherian domain (respectively, locally $S$-Noetherian domain) if and only if the Nagata ring $D[X]_N$ is an $S$-Noetherian domain (respectively, locally $S$-Noetherian domain). In 2018, H. Kim and J. W. Lim \cite{kl18} introduced $S$-$*_w$-principal ideal domains and studied the local property, the Nagata type theorem, and the Cohen type theorem for  $S$-$*_w$-principal ideal domains.

In 2018, D. Bennis and M. El Hajoui\cite{bh18} introduced $S$-finitely presented modules and $S$-coherent rings which are $S$-version of finitely presented modules and coherent rings and obtained an $S$-version of Chase's result \cite[Theorem 2.2]{c} as bellow.
 \begin{theorem}\label{B S-chase}
\textbf{ \cite[Theorem 3.8]{bh18}}
The following assertions are equivalent:
\begin{enumerate}
    \item $R$ is an $S$-coherent ring;
    \item $(I:a)$ is an $S$-finite ideal of $R$, for every finitely generated ideal $I$ of $R$ and $a\in R$;
     \item  $(0:a)$ is an $S$-finite ideal of $R$ for every $a\in R$ and the intersection of  two finitely generated ideals of $R$ is an $S$-finite ideal of $R$.
\end{enumerate}
\end{theorem}

Subsequently, they proposed an interesting open question as an $S$-version of Chase Theorem \cite[Theorem 2.1]{c}.
\begin{open question}
 How to give an $S$-version of flatness that characterizes $S$-coherent rings similarly to the classical case?
\end{open question}

One of the main purposes of this article is to characterize $S$-coherent rings inspired by this question. In Section 2 and Section 3, we introduce and study the notions of  $S$-flat modules and $S$-Mittag-Leffler modules with respect to a class $\Q$ (denoted by $S$-$\mathcal{Q}$-ML).

In Section 4, we study $S$-Mittag-Leffler modules with respect to all flat modules (denoted by $S$-$\mathcal{F}$-ML) and show the following result.\\
{\bf Proposition 4.2.}\ \ {\sl   An $R$-module $M$ is $S$-finitely presented if and only if it is in $S$-$\mathcal{F}$-ML and finitely generated.}\\
Then we give a new characterization of $S$-coherent rings using the class $S$-$\mathcal{F}$-ML.\\
{\bf Theorem 4.3.}\ \ {\sl  The following assertions are equivalent for a ring $R$:
\begin{enumerate}
    \item $R$ is an $S$-coherent ring;
    \item every ideal is in $S$-$\mathcal{F}$-ML;
     \item every finitely generated ideal is in $S$-$\mathcal{F}$-ML;
    \item every submodule of projective modules is in $S$-$\mathcal{F}$-ML;
    \item the class $S$-$\mathcal{F}$-ML is closed under finitely generated submodules;
    \item the class $S$-$\mathcal{F}$-ML is closed under submodules.
\end{enumerate}}
Utilizing this characterization, we finally obtain the following $S$-version of Chase Theorem.\\
{\bf Theorem 4.4.}\textbf{ $($the $S$-version of Chase Theorem$)$}\ \ {\sl The following assertions are equivalent:
\begin{enumerate}
    \item $R$ is an $S$-coherent ring;
    \item any product of flat $R$-modules is $S$-flat;
       \item any product of  projective $R$-modules is $S$-flat;
        \item any product of $R$ is $S$-flat;
\end{enumerate}}

\section{Preliminaries}
In this section, we will investigate some $S$-version of classical definitions on finitely generated modules, finitely presented modules, coherent rings and flat modules. Let $S$ be a multiplicative closed set of $R$.
\begin{definition}{ Let $M$ and $N$ be $R$-modules.

$(1)$ $\tau_S(M)=\{x\in M|sx=0$ for some $s\in S\}$ is called the total $S$-torsion submodule of $M$. If $\tau_S(M)=0$, then $M$ is called an $S$-torsion-free module; If $\tau_S(M)=M$, then $M$ is called an $S$-torsion module.

$(2)$ An $R$-homomorphism $f:M\rightarrow N$ is an $S$-monomorphism (resp. $S$-epimorphism, $S$-isomorphism) if the induced $R_S$-homomorphism $f_S:M_S\rightarrow N_S$ is a monomorphism (resp. an epimorphism, an isomorphism).

$(3)$ A sequence $0\rightarrow M\rightarrow N\rightarrow L\rightarrow 0$ is $S$-exact if the induced sequence $0\rightarrow M_S\rightarrow N_S\rightarrow L_S\rightarrow 0$ is exact.}

\end{definition}

\begin{remark}\label{FPFG} It is easy to verify the following assertions.
\begin{enumerate}
   \item An $R$-homomorphism $f:M\rightarrow N$ is an $S$-monomorphism if and only if $\Ker (f)$ is $S$-torsion.

   \item An $R$-homomorphism $f:M\rightarrow N$ is an $S$-epimorphism if and only if $\Coker (f)$ is $S$-torsion.

   \item The class of $S$-torsion modules is closed under submodules, quotient modules, extensions and directed limits.

\end{enumerate}
\end{remark}

The following definition follows from \cite{bh18}.
\begin{definition}\label{D S-coh}{
$(1)$ An $R$-module $M$ is said to be $S$-finite, if there exists a finitely generated submodule $N$ of $M$ such that $sM\subseteq N$ for some $s\in S$.

$(2)$ An $R$-module $M$ is said to be $S$-finitely presented, if there exists an exact sequence of $R$-modules $0\rightarrow K\rightarrow F\rightarrow M\rightarrow 0$, where $K$ is $S$-finite and $F$ is a finitely generated free $R$-module.

$(3)$ A ring $R$ is $S$-Noetherian provided that every ideal of $R$ is $S$-finite.

$(4)$  A ring $R$ is $S$-coherent provided that every finitely generated ideal of $R$ is $S$-finitely presented.}

\end{definition}

\begin{remark}\label{FPFG}It is easy to verify the following assertions.
\begin{enumerate}
   \item For any multiplicative closed set $S$, every $S$-finitely presented module is finitely generated, every finitely generated module is $S$-finite.

   \item In Definition \ref{D S-coh}(1), since $sM\subseteq N\subseteq  M$, we obtain $N_S=M_S$.

  \item Every $S$-Noetherian ring is an $S$-coherent ring, see \cite[Remark 3.4(1)]{bh18}.

  \item If $R$ is an $S$-coherent ring, then $R_S$ is a coherent ring. Indeed, let  $I_S$ be a finitely generated ideal of $R_S$ such that $I$ is a finitely generated ideal of $R$, then there is an exact sequence $0\rightarrow K\rightarrow R^n\rightarrow I\rightarrow 0$ such that $K$ is $S$-finite. Then $K_S$ is finitely generated as an $R_S$-ideal, and thus $R_S$ is a coherent ring.
\end{enumerate}
\end{remark}

\begin{definition}\label{s-flat}{Let $M$ be an $R$-module, then $M$ is said to be $S$-flat if for any finitely generated ideal $I$ of $R$, the natural homomorphism $I\otimes_RM\rightarrow R\otimes_RM$ is an $S$-monomorphism.
}
\end{definition}

Obviously, every flat module is $S$-flat. However, the converse does not hold. Indeed, let $R$ be a domain not a field, $S$ be the set of nonzero elements in $R$, then every  $R$-module is $S$-flat. Thus there exists some  $S$-flat module which is not flat.

Now,  we give a characterization of $S$-flat modules.

\begin{proposition}\label{c s-flat} Let $M$ be an $R$-module, the following assertions are equivalent:

\begin{enumerate}
   \item $M$ is $S$-flat;
   \item for any finitely generated ideal $I$ of $R$,  $\phi: I\otimes_RM\rightarrow IM$ is an $S$-isomorphism;
   \item $M_S$ is a flat $R_S$-module.
\end{enumerate}
\end{proposition}

\begin{proof}
$(1)\Leftrightarrow (2):$ Let $I$ be a finitely generated ideal of $R$, consider the following commutative diagram,

$$\xymatrix@R=20pt@C=15pt{
  & I\otimes_R M\ar[d]_{\phi }\ar[r]^{f} & R\otimes_R M\ar[d]_{\cong}\\
   0 \ar[r]^{} & IM \ar[r]^{} &RM.\\}$$
We have $M$ is $S$-flat if and only if  $f$ is an $S$-monomorphism if and only if $\phi$ is an $S$-monomorphism.

$(1)\Rightarrow (3):$ Let $I_S$ be a finitely generated ideal of $R_S$, where $I$ is a finitely generated ideal of $R$. Since $M$ is $S$-flat, the natural homomorphism $I\otimes_RM\rightarrow R\otimes_RM$ is an $S$-monomorphism. By localizing at $S$, the natural homomorphism
$$I_S\otimes_{R_S}M_S\cong (I\otimes_RM)_S\rightarrow (R\otimes_RM)_S\cong R_S\otimes_{R_S}M_S$$ is an $R_S$-monomorphism. Thus $M_S$ is a flat $R_S$-module.

$(3)\Rightarrow (1):$ Let $I$ be a finitely generated ideal of $R$, then $I_S$ is a finitely generated ideal of $R_S$. Since $M_S$ is a flat $R_S$-module, the natural homomorphism $(I\otimes_RM)_S\rightarrow (R\otimes_RM)_S$ is an $R_S$-monomorphism. So, $M$ is an $S$-flat module.
\end{proof}

\section{$S$-Mittag-Leffler modules with respect to a class of $R$-modules}
Recall the classical case in \cite{hh}. Let $\Q$ be a class of $R$-modules and $M$ an $R$-module.
We say that $M$ is \emph{Mittag-Leffler} with respect to $\Q$ if the canonical map
\begin{center}
$\phi_{M,\Q}: M\bigotimes_R\prodi Q_i\rightarrow \prodi  (M\bigotimes_R Q_i)$
\end{center}
is a monomorphism for any family $\{Q_i\}_{i\in I}$ of modules in $\Q$.
In case that $\Q$ is the class of all $R$-modules, we say $M$ is Mittag-Leffler. In {\cite[Corollary 4.3]{mc}}, the authors characterized coherent rings using homological properties of Mittag-Leffler modules.  In order to characterize $S$-coherent rings, we introduce and study $S$-Mittag-Leffler modules with respect to a given class $\mathcal{Q}$.

\begin{definition}{ Let $\mathcal{Q}$ be a  class of $R$-modules, $M$ is said  to be an $S$-Mittag-Leffler module with respect to $\mathcal{Q}$, if the natural homomorphism $\phi_{M,\Q}$ is an $S$-monomorphism.}
\end{definition}
We denote $S$-$\mathcal{Q}$-ML to be the class of all $S$-Mittag-Leffler modules with respect to $\mathcal{Q}$. When  $\mathcal{Q}$ is the class $\mathcal{F}$ of all flat modules (resp. the class $\mathcal{P}$ of all  projective $R$-modules), we denote it by $S$-$\mathcal{F}$-ML(resp. $S$-$\mathcal{P}$-ML). Obviouly,  Mittag-Leffler modules with respect to $\mathcal{Q}$ and $S$-torsion modules are all in $S$-$\mathcal{Q}$-ML. The class $S$-$\mathcal{Q}$-ML  is closed under pure submodules, pure extensions. Any direct sum of modules is in  $S$-$\mathcal{Q}$-ML if and only if each direct summand is in  $S$-$\mathcal{Q}$-ML. If $N$ is a finitely generated submodule of $M$ in $S$-$\mathcal{Q}$-ML, then $M/N$ is also in $S$-$\mathcal{Q}$-ML.

Let $\mathcal{C}$ be a class of $R$-modules, if $\mathcal{C}$ is closed under pure submodules, direct products and direct limits, then $\mathcal{C}$ is said to be a \emph{definable class }(see \cite[Theorem 3.4.7]{P09}). Let $\mathcal{Q}$ be a class of $R$-modules. The class $\overline{\mathcal{Q}}$ which denotes the smallest definable class containing $\mathcal{Q}$ is said to be the definable closure of $\mathcal{Q}$. Note that $\overline{\mathcal{Q}}$ can be constructed by closing $\mathcal{Q}$ under direct products, then under pure submodules and finally under directed limits. For the sake of simplification, we use the following symbols,

\begin{tabular}{ll}
$\lim\limits_{\longrightarrow }\mathcal{Q}$: &all directed limits of modules in $\mathcal{Q}$,\\
$Prod\mathcal{Q}$: &all direct products of modules in $\mathcal{Q}$,\\
$Psub\mathcal{Q}$: &all pure submodules of modules in $\mathcal{Q}$,\\
$\overline{\mathcal{Q}}$:  & the definable closure of $\mathcal{Q}$.
\end{tabular}

\begin{proposition}\label{w-M}
Let $M$ be an $R$-module, then the following assertions are equivalent:
\begin{enumerate}
    \item $M\in\wMLQ$;
   \item$M\in\MLPsubQ$;
    \item$M\in\MLProdQ$;
   \item$M\in\MLlimQ$;
    \item$M\in\MLcQ$.
\end{enumerate}
\end{proposition}
\begin{proof}
$(2)\Rightarrow(1)$, $(3)\Rightarrow(1)$, $(4)\Rightarrow(1)$ and $(5)\Rightarrow(1)$ are obvious.

$(1)\Rightarrow(2)$: For any subset $\{Q^{\prime}_i\}_{i\in I}$ from the class $Psub\mathcal{Q}$, there exists a family $\{Q_i\}_{i\in I}$ from $\mathcal{Q}$, such that $Q^{\prime}_i$ is a pure submodule of $Q_i$, for each $i\in I$. There exist a commutative diagram,
$$\xymatrix@R=30pt@C=50pt{
 0 \ar[r]^{}  & M\bigotimes_R \prodi Q^{\prime}_i\ar[d]_{\phi_{M}}\ar[r] &M\bigotimes_R \prodi Q_i\ar[d]_{\phi_{M,\Q}}\\
 0 \ar[r]^{}  &\prodi( M\bigotimes_R  Q^{\prime}_i )\ar[r] &\prodi (M\bigotimes_R Q_i)}$$
Since $\phi_{M,\Q}$ is an $S$-monomorphism, $\phi_M$ is also an $S$-monomorphism. That is, $M\in\MLPsubQ$.

$(1)\Rightarrow(3)$:  For any subset  $\{Q^{\prime}_i\}_{i\in I}$ from $Prod\mathcal{Q}$, there exists a family $\{Q_{i,j}\}_{i\in I,j\in J}$ from $\mathcal{Q}$ , such that $Q^{\prime}_i=\prod_{j\in J}Q_{i,j}$. There is a commutative diagram
$$\xymatrix@R=20pt@C=10pt{
M\bigotimes_R\prodi Q^{\prime}_i \ar[rr]^{\phi_{M,\Q}}\ar@{->}[rd]^{\phi_{M}} &&\prodij  (M\bigotimes_R Q_{i,j}) \\
    & \prodi(M\bigotimes_R Q^{\prime}_i)\ar[ru]&&    .\\}$$
Since $\phi_{M,\Q}$ is an $S$-monomorphism, $\phi_{M}$ is an $S$-monomorphism. Consequently, $M\in\MLProdQ$.

$(1)\Rightarrow(4)$:  For any subset $\{Q^{\prime}_i\}_{i\in I}$ from $\lim\limits_{\longrightarrow }\mathcal{Q}$, we have $Q^{\prime}_i={\lim\limits_{\longrightarrow }}(Q_{\alpha}^i,f_{\beta\alpha}^i)_{\beta,\alpha\in J_i}$ with $Q_{\alpha}^i\in \Q$ for any $\alpha\in J_i$. Suppose $f_{\alpha}^i:Q_{\alpha}^i\rightarrow Q^{\prime}_i$ is the canonical homomorphism. Next we will show $\phi_{M}: M\bigotimes_R\prodi Q^{\prime}_i\rightarrow \prodi  (M\bigotimes_R Q^{\prime}_i)$ is an $S$-homomorphism. Let $y=\sum_{j=1}^n x_j'\otimes(q_{j,i})_{i\in I}\in \Ker \phi_{M}$, we have $\sum_{j=1}^n x_j'\otimes q_{j,i}=0$ for any $i\in I$. Thus, for any $i\in I$ there exist $\alpha_i\in  I_i$ and $k_{1,i},....,k_{n,i}\in Q_{\alpha_i,i}$ such that $\sum_{j=1}^n x_j'\otimes k_{j,i}=0$ in $ M\otimes_R Q_{\alpha_i,i}$ and $f_{\alpha_i,i}(k_{j,i})=q_{j,i}$ for any $j=1,...,n$. We have the following commutative diagram

$$\xymatrix@R=30pt@C=50pt{
 M\bigotimes_R \prodi Q_{\alpha_i,i}\ar[d]_{\phi_{M,\Q}}\ar[r]^{M\otimes \prod_{i\in I}f_{\alpha_i}^i} &M\bigotimes_R \prodi Q^{\prime}_i\ar[d]_{\phi_{M}}\\
 \prodi( M\bigotimes_R  Q_{\alpha_i,i} )\ar[r]_{\prod_{i\in I}(M\otimes f_{\alpha_i}^i)} &\prodi (M\bigotimes_R Q^{\prime}_i).}$$
By construction, we obtain
$$y=\sum_{j=1}^n x_j'\otimes(q_{j,i})_{i\in I}=(M\otimes \prod_{i\in I}f_{\alpha_i}^i)(\sum_{j=1}^n x_j'\otimes(k_{j,i})_{i\in I})$$
and
$$\phi_{M,\Q}(\sum_{j=1}^n x_j'\otimes(k_{j,i})_{i\in I})=(\sum_{j=1}^n x_j'\otimes k_{j,i})_{i\in I}=0.$$  Since $\phi_{M,\Q}$ is an $S$-monomorphism, there exists $s\in S$ such that $s(\sum_{j=1}^n x_j'\otimes (k_{j,i})_{i\in I})=0$. Then $sy=s(\sum_{j=1}^n x_j'\otimes(q_{j,i})_{i\in I})=s(M\otimes \prod_{i\in I}f_{\alpha_i}^i)(\sum_{j=1}^n x_j'\otimes(k_{j,i})_{i\in I})=(M\otimes \prod_{i\in I}f_{\alpha_i}^i)(s(\sum_{j=1}^n x_j'\otimes(k_{j,i})_{i\in I}))=0$. Then $\Ker \phi_{M}$ is $S$-torsion and thus $\phi_{M}$ is an $S$-monomorphism.

$(2)+(3)+(4)\Rightarrow(5)$ is obvious.
\end{proof}

\begin{corollary}\label{PFRMLP}
Let $M$ be an $R$-module, then $M\in\wMLF$ if and only if $M\in\wMLP$ if and only if $M\in\wMLR$.
\end{corollary}
\begin{proof}
Since any projective module is pure submodule of direct product of $R$ and any flat module is direct limit of projective modules, the definable closures of $\{R\}$ and of all projectives and all are flats are the same thing. Thus the consequence holds from Proposition \ref{w-M}.
\end{proof}

For a class $\mathcal{T}$ of $R$-modules, we denote $\mathcal{T}^{\top}$ the class  of all $R$-modules $M$ such that
$\Tor_1^R(T,M)=0$ for any $T\in \mathcal{T}$. Let $M$ be an $R$-module, $\mathcal{C}$ a class of $R$-modules, $\tau$ an ordinal. An increasing chain $(M_{\alpha}|\alpha\leq \tau)$ of submodules of $M$ is a $\mathcal{C}$-filtration of $M$ provided that $M_0=0$, $M_{\tau}=M$, $M_{\alpha}=\bigcup_{\beta<\alpha}M_{\beta}$ for a limit ordinal $\alpha$ and $M_{\beta+1}/M_{\beta}\in \mathcal{C}$ for any $\beta<\tau$. The following result which is similar to {\cite[Proposition 1.9]{hh}} is crucial to the study of $S$-Mittag-Leffler modules with respect to all flat modules.

\begin{proposition}\label{filtration prop}
Let $\mathcal{T}$ be a class of $R$-modules that are $S$-Mittag-Leffler with respect to $\mathcal{Q}\subseteq \mathcal{T}^{\top}$, then any module isomorphic to a direct summand of a $\mathcal{T} \bigcup \mathcal{P}$-filtered module is an $S$-Mittag-Leffler module with respect to $\mathcal{Q}$.
\end{proposition}

\begin{proof}

We imitate the proof given by {\cite[Proposition 1.9]{hh}} with some changes. We can also assume that $\mathcal{T}$ contains $\mathcal{P}$. Indeed, all projective modules are in $S$-$\mathcal{Q}$-ML and $(\mathcal{T}\bigcup \mathcal{P})^{\top}= \mathcal{T}^{\top}$.

Let $M$ be a $\mathcal{T}$-filtered $R$-module, $\tau$ be an ordinal such that there exists an $\mathcal{T}$-filtration $\{M_{\alpha}|\alpha\leq \tau\}$. We prove by induction that $M_{\alpha}$ are all  in $S$-$\mathcal{Q}$-ML.

Firstly, let $\alpha$ be a successor ordinal, $\{Q_i|i\in I\}$ be a set of modules in $\mathcal{Q}$. Consider the exact sequence $0\rightarrow M_{\alpha-1}\rightarrow  M_{\alpha}\rightarrow M_{\alpha}/M_{\alpha-1}\rightarrow 0$, we obtain the following commutative diagram of exact sequences.
$$\xymatrix@R=25pt@C=10pt{
  & M_{\alpha-1}\otimes_R \prod_{i\in I} Q_i\ar[d]_{\phi_{M_{\alpha-1}}}\ar[r]^{} & M_{\alpha}\otimes_R  \prod_{i\in I} Q_i \ar[d]_{\phi_{M_{\alpha}}}\ar[r]^{} & M_{\alpha}/M_{\alpha-1}\otimes_R  \prod_{i\in I} Q_i \ar[d]_{\phi_{M_{\alpha}/M_{\alpha-1}}}\ar[r]^{} &  0\\
0 \ar[r]^{} & \prod_{i\in I} (M_{\alpha-1}\otimes_R Q_i )\ar[r]^{f} &\prod_{i\in I} (M_{\alpha}\otimes_R Q_i)  \ar[r]^{} & \prod_{i\in I}( M_{\alpha}/M_{\alpha-1}\otimes_R Q_i )\ar[r]^{} &  0\\}$$

Note that $f$ is an monomorphism as $M_{\alpha}/M_{\alpha-1}\in \mathcal{T}\subseteq \mathcal{Q}^{\top}$. Since $\phi_{M_{\alpha-1}}$ and $\phi_{M_{\alpha}/M_{\alpha-1}}$ are $S$-monomorphisms, we have $\phi_{M_{\alpha}}$ is an $S$-monomorphism.

Secondly, let $\alpha$ be a limit ordinal such that $M_{\beta}$ be in $S$-$\mathcal{Q}$-ML for any $\beta<\alpha$, we show that $M_{\alpha}=\bigcup_{\beta<\alpha}M_{\beta}$ is $S$-Mittag-Leffler with respect to $\mathcal{Q}$. Let $\{Q_i|i\in I\}$ be a set of modules in $\mathcal{Q}$ and $x\in \Ker\phi_{M_{\alpha}}$, then there exists $\beta<\alpha$ and $y\in M_{\beta}\otimes_R \prod_{i\in I} Q_i$ such that $ x=(\varepsilon_{\beta}\otimes_R \prod_{i\in I} Q_i)(y)$, where $\varepsilon_{\beta}:M_{\beta}\rightarrow M_{\alpha}$ is the natural monomorphism. Consider the following commutative diagram,

$$\xymatrix@R=35pt@C=60pt{
 M_{\beta}\otimes_R \prod_{i\in I} Q_i\ar[d]_{\phi_{M_{\beta}}}\ar[r]^{\varepsilon_{\beta}\otimes_R \prod_{i\in I} Q_i} & M_{\alpha}\otimes_R  \prod_{i\in I} Q_i \ar[d]_{\phi_{M_{\alpha}}}\\
\prod_{i\in I} (M_{\beta}\otimes_R Q_i) \ar@{^{(}->}[r]^{\prod_{i\in I} (\varepsilon_{\beta}\otimes_R  Q_i)} &\prod_{i\in I} (M_{\alpha}\otimes_R Q_i ).\\}$$
Note that $\phi_{M_{\beta}}$ is an $S$-monomorphism and $\prod_{i\in I} (\varepsilon_{\beta}\otimes_R  Q_i) $ is a monomorphism since $\Tor_1^R(M_{\alpha}/M_{\beta},Q_i)=0$. Thus $y\in \Ker\ (\phi_{M_{\beta}}\circ \prod_{i\in I} (\varepsilon_{\beta}\otimes_R  Q_i))$ is $S$-torsion, and there is some $s\in S$ such that $sy=0$.
Therefore, $sx=s(\varepsilon_{\beta}\otimes_R \prod_{i\in I} Q_i(y))=(s\varepsilon_{\beta}(y))\otimes_R \prod_{i\in I} Q_i=(\varepsilon_{\beta}(sy))\otimes_R \prod_{i\in I} Q_i=0$.  Consequently, $\phi_{M_{\alpha}}$ is an $S$-monomorphism.

\end{proof}

\begin{corollary}\label{w-MLP}
The following statements hold.
\begin{enumerate}
    \item The class $\wMLF$ is closed under $\wMLF$-filtration.
    \item Let $(M_{\alpha}|\alpha\leq \tau)$ be a chain with each $M_{\alpha} \in \wMLF$. Then  $M=\bigcup_{\beta<\alpha}M_{\beta}$ is also in $\wMLF$.
\end{enumerate}
\end{corollary}

\begin{proof}
(1) Since any flat module is in $\mathcal{T}^{\top}$, (1) follows from Proposition \ref{filtration prop} immediately by  putting $\Q=\mathcal{F}$ and $\mathcal{T}=\wMLF$.

(2) If $\alpha$ is a successor ordinal, then $M=M_{\alpha-1}$. If $\alpha$ is a limit ordinal, then the proof of Proposition \ref{filtration prop} on the limit case remains valid.
\end{proof}

\section{Characterizing $S$-Coherent rings using $S$-$\mathcal{F}$-ML }
In this section, we utilize the class $S$-$\mathcal{F}$-ML of $S$-Mittag-Leffler modules with respect to all flat modules to characterize $S$-Coherent rings, and then we give an $S$-version of Chase theorem using $S$-flat modules as an application. Firstly, we build a connection among  finitely generated modules, $S$-Mittag-Leffler modules with respect to all flat modules and $S$-finitely presented modules.
\begin{lemma}\label{S-epi}
Let $\mathcal{F}$ be a set of flat modules containing $R$, $\phi_{M,\mathcal{F}}$ the natural homomorphism.
\begin{enumerate}
    \item If $\phi_{M,\mathcal{F}}$ is an $S$-epimorphism, then $M$ is $S$-finite.
    \item Moreover, if $M$ is finitely generated and $\phi_{M,\mathcal{F}}$ is an $S$-isomorphism, then $M$ is $S$-finitely presented.
\end{enumerate}
\end{lemma}
\begin{proof}
$(1)$ If $\phi_{M,\mathcal{F}}$ is an $S$-epimorphism, we consider the exact sequence
$$\xymatrix{
M\otimes_R  R^M \ar[rr]^{\phi_{M,\mathcal{F}}}\ar@{->>}[rd] &&M^M \ar[r]^{} & T\ar[r]^{} &  0\\
    &\Im \phi_{M} \ar@{^{(}->}[ru] &&  &   \\}$$
with $T$ an $S$-torsion module. Let $x=(m)_{m\in M}\in M^M$, there is some $s\in S$ such that $sx\in \Im \phi_{M,\mathcal{F}}$. Subsequently, for any $i\in M$, there exists $m_j\in M, r_{j,i}\in R$ such that $sx=\phi_{M,\mathcal{F}}(\sum_{j=1}^n m_j\otimes (r_{j,i})_{i\in M})=(\sum_{j=1}^n m_j r_{j,i})_{i\in M}$. Set $K=\langle\{m_j|j=1,....,n \}\rangle$ be a finitely generated submodule of $M$. Now, for any $m\in M$, $sm=\sum_{j=1}^n m_j r_{j,m}\in K$, thus $sM\subseteq K\subseteq M$ and then $M$ is $S$-finite.

$(2)$ Let $0\rightarrow K\rightarrow F\rightarrow M\rightarrow 0$ be an exact sequence, where $F$ is a finitely generated free $R$-module. Consider the following commutative diagram of exact sequences,
 $$\xymatrix{
& K\otimes_R \prod_{i\in I} R\ar[d]_{\phi_{K,\mathcal{F}}}\ar[r]^{} & F\otimes_R  \prod_{i\in I} R \ar[d]_{\phi_{F,\mathcal{F}}}\ar[r]^{} & M\otimes_R  \prod_{i\in I} R \ar[d]_{\phi_{M,\mathcal{F}}}\ar[r]^{} &  0\\
   0 \ar[r]^{} & \prod_{i\in I} (K\otimes_R R )\ar[r]^{} &\prod_{i\in I} (F\otimes_R R ) \ar[r]^{} & \prod_{i\in I} (M\otimes_R R ) \ar[r]^{} &  0\\}$$
Since  $\phi_{F,\mathcal{F}}$ is an isomorphism and $\phi_{M,\mathcal{F}}$  is an $S$-isomorphism, then $\phi_{K,\mathcal{F}}$ is an $S$-epimorphism, thus $K$ is $S$-finite by (1).

\end{proof}

\begin{proposition}\label{finite}
An $R$-module $M$ is $S$-finitely presented if and only if  it is in $S$-$\mathcal{F}$-ML and finitely generated.
\end{proposition}
\begin{proof}
For the ``only if'' part, $M$ is finitely generated by Remark \ref{FPFG} (1). Now, we show $M$ is in $S$-$\mathcal{F}$-ML. Let $0\rightarrow K\rightarrow F\rightarrow M\rightarrow 0$ be an exact sequence, where $K$ is $S$-finite and $F$ is a finitely generated free $R$-module. Consider the following commutative diagram of exact sequences,
 $$\xymatrix{
& K\otimes_R \prod_{i\in I} R\ar[d]_{\phi_{K,\mathcal{F}}}\ar[r]^{} & F\otimes_R  \prod_{i\in I} R \ar[d]_{\cong}\ar[r]^{} & M\otimes_R  \prod_{i\in I} R \ar[d]_{\phi_{M,\mathcal{F}}}\ar[r]^{} &  0\\
   0 \ar[r]^{} & \prod_{i\in I} (K\otimes_R R )\ar[r]^{} &\prod_{i\in I} (F\otimes_R R ) \ar[r]^{} & \prod_{i\in I} (M\otimes_R R ) \ar[r]^{} &  0.\\}$$

To prove  $\phi_{K,\mathcal{F}}$ is an $S$-monomorphism, we only need to show  $\phi_{K,\mathcal{F}}$ is an $S$-epimorphism. Since $K$ is $S$-finite, there exists a finitely generated submodule $K'$ of $K$ such that $sK\subseteq K'\subseteq K$ for some $s\in S$. The natural commutative diagram

$$\xymatrix{
 K'\otimes_R \prod_{i\in I} F_i     \ar[d]_{\phi_{K,\mathcal{F}}}\ar[r]^{} &     K\otimes_R  \prod_{i\in I} F_i \ar[d]_{\phi_{K,\mathcal{F}}}\\
 \prod_{i\in I} (K'\otimes_R F_i )\ar[r]^{}                            &     \prod_{i\in I} (K\otimes_R F_i )  \\ }$$
induces the following commutative diagram by localizing at $S$,
$$\xymatrix{
 K'_S\otimes_R \prod_{i\in I} F_i\ar@{->>}[d]_{\phi_{K,\mathcal{F}}^S}\ar[r]^{\cong} &     K_S\otimes_R  \prod_{i\in I} F_i\ar[d]_{\phi_{K,\mathcal{F}}^S}\\
 (\prod_{i\in I} (K'\otimes_R F_i ))_S\ar[r]^{f}                            &     (\prod_{i\in I} (K\otimes_R F_i ))_S.\\}$$
For any $k_i\in K, q_i\in F_i(i\in I)$  and $t\in S$, we obtain $\frac{(k_i\otimes_Rq_i)_{i\in I}}{t}=\frac{s(k_i\otimes_Rq_i)_{i\in I}}{st} =\frac{(sk_i\otimes_Rq_i)_{i\in I}}{st}  \in (\prod_{i\in I} (K'\otimes_R F_i ))_S$, thus  $f$ is an epimorphism (moreover, since  $f$ is a monomorphism, $f$ is an isomorphism). Thus $\phi_{K,\mathcal{F}}^S$ is an epimorphism, and then $\phi_{K,\mathcal{F}}$ is an $S$-epimorphism.

For the ``if'' part, let $M$ be a  finitely generated $R$-module in $S$-$\mathcal{F}$-ML. Consider the exact sequence $0\rightarrow K\rightarrow F \rightarrow M\rightarrow 0$ with $F$ a finitely generated free module, then there is a commutative diagram of exact sequences,
  $$\xymatrix{
& K\otimes_R \prod_{i\in I} F_i\ar[d]_{\phi_{K,\mathcal{F}}}\ar[r]^{} & F\otimes_R  \prod_{i\in I} F_i \ar[d]_{\phi_{F,\mathcal{F}}}\ar[r]^{} & M\otimes_R  \prod_{i\in I} F_i \ar[d]_{\phi_{M,\mathcal{F}}}\ar[r]^{} &  0\\
   0 \ar[r]^{} & \prod_{i\in I} (K\otimes_R F_i )\ar[r]^{} &\prod_{i\in I} (F\otimes_R F_i ) \ar[r]^{} & \prod_{i\in I} (M\otimes_R F_i ) \ar[r]^{} &  0.\\}$$
Note that $\phi_{M,\mathcal{F}}$ is an $S$-monomorphism and $\phi_{F,\mathcal{F}}$ is an isomorphism, thus $\phi_{K,\mathcal{F}}$ is an $S$-epimorphism. Then $K$ is $S$-finite by Lemma \ref{S-epi}(1) and $M$ is $S$-finitely presented.

\end{proof}

In {\cite[Corollary 4.3]{mc}}, Izurdiaga proved a ring $R$ is coherent if and only if the class of  Mittag-Leffler $R$-modules with respect to all flat modules  is closed under submodules. Similarly, we give an $S$-version of of Izurdiaga's result on $S$-coherent rings.
\begin{theorem}\label{S-coh}
The following assertions are equivalent for a ring $R$:
\begin{enumerate}
    \item $R$ is an $S$-coherent ring;
    \item every ideal is in $S$-$\mathcal{F}$-ML;
     \item every finitely generated ideal is in $S$-$\mathcal{F}$-ML;
    \item every submodule of projective modules is in $S$-$\mathcal{F}$-ML;
    \item the class $S$-$\mathcal{F}$-ML is closed under finitely generated submodules;
    \item the class $S$-$\mathcal{F}$-ML is closed under submodules.
\end{enumerate}
\end{theorem}

\begin{proof}

$(2)\Rightarrow(3)$, $(6)\Rightarrow(4)$ , $(4)\Rightarrow(2)$,  $(6)\Rightarrow(5)$ and $(5)\Rightarrow(3)$ are obvious. $(3)\Leftrightarrow(1)$: By Proposition \ref{finite}.

$(4)\Rightarrow(6)$: Let $M$ be an $R$-module in $S$-$\mathcal{F}$-ML, $K$ a submodule of $ M$ and consider the following pull-back diagram,
$$\xymatrix@R=20pt@C=25pt{ & & 0\ar[d]&0\ar[d]&\\
& & K'\ar[d]\ar@{=}[r]&K'\ar[d]&\\
0 \ar[r]^{} & K\ar@{=}[d]\ar[r]^{} & Q \ar[d]\ar[r]^{} &P\ar[d]\ar[r]^{} &  0\\
0 \ar[r]^{} & K\ar[r]^{} & M \ar[d]\ar[r]^{} &M/K\ar[d]\ar[r]^{} &  0\\
& & 0 &0 &\\}$$
with $P$ projective. By (4), $K'$ is in $S$-$\mathcal{F}$-ML and thus $Q$ is in $S$-$\mathcal{F}$-ML. As the middle row split, $K$  is also in  $S$-$\mathcal{F}$-ML.

$(2)\Rightarrow(4)$: As $S$-$\mathcal{F}$-ML is closed under direct summand, we just consider the free module case. Assume $L$ is a submodule of $R^{(\beta)}$, where $\beta$ is an ordinal, we prove (4) by induction on $\beta$.

If  $\beta$ is a successor ordinal, consider the exact sequence\begin{center}
$0\ra L\bigcap R^{(\beta-1)}\ra L\bigcap R^{(\beta)}\ra J\ra 0$
\end{center}
for some ideal $J$ of $R$. Since the ideal $J$ is in $S$-$\mathcal{F}$-ML by (2) and $L\bigcap R^{(\beta-1)}$ is in $S$-$\mathcal{F}$-ML by induction, thus $L=L\bigcap R^{(\beta)}$  is in $S$-$\mathcal{F}$-ML.

If  $\beta$ is a limit ordinal, then $L=\bigcup_{\gamma<\beta}(L\cap R^{(\gamma)})\in S$-$\mathcal{F}$-ML by $L$ is a direct unions of modules in $S$-$\mathcal{F}$-ML.

$(3)\Rightarrow(2)$: By {\cite[Corollary 3.6]{mc}}, we just need to prove any $(\aleph_0,$ $S$-$\mathcal{F}$-ML)-free module (which is a direct union of modules in $S$-$\mathcal{F}$-ML) belong to  $S$-$\mathcal{F}$-ML. This can exactly be deduced from Corollary \ref{w-MLP}(2).
\end{proof}

In {\cite[Theorem 2.1]{c}}, Chase proved that a ring $R$ is coherent if and only if any product of $R$ is flat if and only if any product of flat $R$-modules is  flat. We extend it to the $S$-version and obtain the promised result.
\begin{theorem}\label{S-chase}
\textbf{$($the $S$-version of Chase Theorem$)$}
The following assertions are equivalent:
\begin{enumerate}
    \item $R$ is an $S$-coherent ring;
    \item any product of flat $R$-modules is $S$-flat;
       \item any product of  projective $R$-modules is $S$-flat;
        \item any product of $R$ is $S$-flat;
\end{enumerate}
\end{theorem}

\begin{proof}
$(1)\Rightarrow(2)$: Let $R$ be an $S$-coherent ring,  $F_i$ be flat $R$-modules, $J$ a finitely generated ideal, we have $J\in S$-$\mathcal{F}$-ML by Theorem \ref{S-coh}. Consequently, the natural homomorphism $J\otimes_R (\prod_{i\in I} F_i)\twoheadrightarrow \prod_{i\in I} J\otimes F_i\cong \prod_{i\in I} (J F_i)=J(\prod_{i\in I} F_i)$ is an $S$-isomorphism, thus $\prod_{i\in I} F_i$ is an $S$-flat module from Proposition \ref{c s-flat}.

$(2)\Rightarrow(3)\Rightarrow(4)$ are obvious.

$(4)\Rightarrow(1)$:  Let $J$ be a finitely generated ideal of $R$, $0\rightarrow J\rightarrow R\rightarrow R/J\rightarrow 0$ an exact sequence. Consider the following commutative diagram,

 $$\xymatrix{
&J\otimes_R \prod_{i\in I} R\ar[d]_{\phi_{J,\mathcal{R}}}\ar[r]^{f} & R\otimes_R  \prod_{i\in I}R \ar[d]_{\phi_{R,\mathcal{R}}}^{\cong}\ar[r]^{} & R/J\otimes_R  \prod_{i\in I} R \ar[d]_{\phi_{M,\mathcal{R}}}^{\cong}\ar[r]^{} &  0\\
   0 \ar[r]^{} & \prod_{i\in I} (J\otimes_RR)\ar[r]^{} &\prod_{i\in I} (R\otimes_RR ) \ar[r]^{} & \prod_{i\in I} (R/J\otimes_RR ) \ar[r]^{} &  0.\\}$$
Since $\prod_{i\in I} R$ is an $S$-flat module, then $f$ is an $S$-monomorphism. Thus $\Ker (f)=\Ker (\phi_{J,\mathcal{R}})$ is $S$-torsion and then $\phi_{J,\mathcal{R}}$ is an $S$-monomorphism and thus $J\in$ $S$-$\mathcal{F}$-ML from Corollary \ref{PFRMLP}.  Consequently, $R$ is $S$-coherent from Theorem \ref{S-coh}.
\end{proof}

\end{document}